\documentclass[10pt]{amsart}

\usepackage{amsmath,amsthm,amsfonts,latexsym,amssymb,mathrsfs,setspace,enumerate,verbatim}

\newtheorem{thm}{Theorem}

\newtheorem{lem}[thm]{Lemma}

\newtheorem{cor}[thm]{Corollary}
\numberwithin{thm}{section}
\numberwithin{equation}{section}

\theoremstyle{definition}

\newtheorem{conj}[thm]{Conjecture}

\newcommand{\real}{\mathbb R}
\newcommand{\com}{\mathbb C}

\newcommand{\intg}{\mathbb Z}

\begin{document}

\title[The areal Mahler measure]{Two inequalities on the areal Mahler measure}
\author[C.L. Samuels \and K.K. Choi]{Charles L. Samuels \and Kwok-Kwong Stephen Choi}
\address{Simon Fraser University, Department of Mathematics, 8888 University Drive, Burnaby, BC V5A 1S6, Canada}
\address{University of British Columbia, Department of Mathematics, 1984 Mathematics Road, Vancouver, BC V6T 1Z2, Canada}
\thanks{Research of all authors is supported in part by NSERC of Canada}
\date{\today}

\begin{abstract}
Recent work of Pritsker defines and studies an areal version of the Mahler measure.  We further explore this function with a particular focus
on the case where its value is small, as this is most relevant to Lehmer's conjecture.  
In this situation, we provide improvements to two inequalities established in Pritsker's original paper.
\end{abstract}

\maketitle

\section{Introduction}

For a polynomial $P(z)$ with complex coefficients, the {\it (logarithmic) Mahler measure} of $P$ is given by
\begin{equation*}
	\log M(P) = \frac{1}{2\pi} \int_{0}^{2\pi} \log|P(e^{it})|dt.
\end{equation*}
In view of the well-known identity
\begin{equation} \label{HardyMahler}
	\log M(P) = \lim_{p\to 0^+}\left( \frac{1}{2\pi} \int_{0}^{2\pi} |P(e^{it})|^pdt\right)^{1/p},
\end{equation}
the Mahler measure is often called the $H^0$ Hardy space norm.

If $P(z) = a_N\prod_{n=1}^N(z-z_n)$ and $P(0)\ne 0$ then Jensen's formula implies that $\log M(P) = \log |a_N| + \sum _{|z_n|>1}\log |z_n|$, and we easily deduce the formula
\begin{equation} \label{MeasureJensen}
	\log M(P) = \log|a_0| - \sum_{|z_n|<1}\log |z_n|,
\end{equation}
where $a_0$ is the constant term of $P$.  It is immediately obvious from this formula that $\log M(P)\geq \log|a_N|$ for all $P\in\com[z]$.  
If we further assume that $P\in\intg[z]$, this inequality yields that $\log M(P) \geq 0$.
Moreover, Kronecker's Theorem establishes that $\log M(P) = 0$ if and only if $P$ is a product of cyclotomic polynomials and $\pm z$.   

As part of a famous 1933 article constructing
large prime numbers, D.H. Lehmer \cite{Lehmer} asked whether there exists a constant $c>0$ such that $\log M(P) \geq c$ in all other cases.  In his investigations, he noted that
\begin{equation*}
	\ell(z) = z^{10} + z^9 -z^7 -z^6 -z^5 -z^4-z^3+z+1
\end{equation*}
has Mahler measure equal to $0.1623\ldots$, although he found no polynomial of measure smaller than this.  Additional work performed since that time provides good
evidence that such a constant $c$ does exist and that $c=0.1623\ldots$.

\begin{conj}[Lehmer's Conjecture]
	There exists $c>0$ such that $\log M(P) \geq c$ whenever $P\in\intg[z]$ is not a product of cyclotomic polynomials and $\pm z$.
\end{conj}

Pritsker \cite{Pritsker} defined the {\it areal Mahler measure} of $P\in\com[z]$ by
\begin{equation*}
	\log\|P\|_0 = \frac{1}{\pi} \iint_{D} \log |P(z)|dA,
\end{equation*}
and it is noted that $\|P\|_0$ satisfies an analogue of \eqref{HardyMahler}
\begin{equation*}
	\log\|P\|_0 = \lim_{p\to 0^+}\left( \frac{1}{\pi} \iint_{D} |P(z)|^p dA\right)^{1/p}.
\end{equation*}
Although the use of Jensen's formula is not obvious here, Pritsker established an analogue of \eqref{MeasureJensen}
\begin{equation} \label{PritskerAlpha}
	\log\|P\|_0 = \log|a_0| + \sum_{|z_n|<1}\left(\frac{|z_n|^2-1}{2}-\log |z_n|\right),
\end{equation}
from which we obtain (see Corollary 1.2 in \cite{Pritsker})
\begin{equation} \label{TrivialInequalities}
	-\frac{\deg P}{2} + \log M(P) \leq \log\|P\|_0 \leq \log M(P).
\end{equation}
Equality occurs in the first inequality precisely when $P$ has no zeros except at the origin, and we obtain equality in the second inequality 
whenever $P$ has no zeros inside the open unit disk.   If we wish to study Lehmer's conjecture, then we will most often consider polynomials
of small Mahler measure.  In this case, the left hand side of \eqref{TrivialInequalities} is negative, and therefore, weaker than even the most trivial bound that $\|P\|_0\geq |a_0|$.

We anticipate that further improvements to \eqref{TrivialInequalities} are possible when $M(P)$ is small and $P\in\intg[z]$.  Indeed, 
examining the invidual terms of the sums in \eqref{MeasureJensen} and \eqref{PritskerAlpha}, we find the power series expansions
\begin{equation*}
	-\log x = \sum_{k=1}^\infty \frac{(1-x)^k}{k} \quad\mathrm{and}\quad \frac{x^2-1}{2} - \log x = (1-x)^2 + 
		\sum_{k=3}^\infty\frac{(1-x)^k}{k}.
\end{equation*}
Of course, this implies that
\begin{equation*}
	\log^2 x = (1-x)^2 + \sum_{k=3}^\infty \left(\sum_{\ell = 1}^{k-1}\frac{1}{\ell(k-\ell)}\right) (1-x)^k,
\end{equation*}
so that the expansions for $\log^2 x$ and $(x^2-1)/2 - \log x$ have the same main terms.  In view of these observations,
we find it reasonable to guess that $\log^2 M(P)$ behaves similarly to $\log \|P\|_0$ when these values are close to zero.

Nevertheless, these remarks are only heuristic and must be made precise.  In this article, we provide upper and lower bounds on 
$\log^2M(P)$ in terms
of $\log \|P\|_0$.  In both cases, our results constitute improvements to \eqref{TrivialInequalities} when $P\in\intg[z]$ and $\|P\|_0$ 
is sufficiently small.  Our first result is the lower bound.

\begin{thm} \label{MainUpper}
	If $P\in\com[z]$ is such that $|P(0)| = 1$ then $\log\|P\|_0 \leq \log^2 M(P)$.
\end{thm}

If $P$ has integer coefficients and $P(0)\ne 0$, then $\|P\|_0< 2$ implies automatically that $|P(0)| = 1$. Furthermore, if $M(P) \leq e$ then $\log^2 M(P) \leq \log M(P)$, 
so that Theorem \eqref{MainUpper} does indeed provide an improvement to the upper bound in \eqref{TrivialInequalities}.

Pritsker further notes that the analogue of Lehmer's conjecture is false for $\|P\|_0$, providing the example $P_n(z) = nz^n-1$.
Here, we have that
\begin{equation*}
	\log M(P_n) = \log n\quad\mathrm{and}\quad \log\|P_n\|_0 = O\left(\frac{\log^2 n}{n}\right).
\end{equation*}
Now let $Z(P)$ denote the number of roots of $P$ that lie inside the open unit disk.
In this example, and in another provided by Pritsker, $Z(P_n)$ tends 
to $\infty$ and $n\to \infty$.  

For the remainder of this article, we will always assume that $Z(P) \geq 1$.  Otherwise Pritsker's work \cite{Pritsker} establishes that
$M(P) = \|P\|_0 = |P(0)|$ and there is very little new information to discover.  Our next result gives an upper bound on $\log^2 M(P)/Z(P)$ in terms of $\|P\|_0$.

\begin{thm} \label{MainLower}
	If $P\in\com[z]$ is such that $|P(0)| = 1$ and $\|P\|_0 < e$ then
	\begin{equation} \label{MainLowerIneq}
		\frac{\log^2 M(P)}{Z(P)} \leq \log\|P\|_0\left(1+ \sum_{k=1}^\infty\left(\sum_{\ell=1}^{k+1}\frac{1}{\ell(k+2-\ell)} - \frac{1}{k+2}\right)\log^{k/2}\|P\|_0\right).
	\end{equation}
\end{thm}

It can be easily verified that $\sum_{\ell=1}^{k-1}1/(\ell(k-\ell)) \leq 1$ so that the interior sum on the right hand side of \eqref{MainLowerIneq} converges.
Moreover, if we are willing to sacrifice a small amount of sharpness, we obtain a more concrete result.

\begin{cor} \label{MainLowerCor}
	If $P\in\com[z]$ is such that $|P(0)| = 1$ and $\|P\|_0 < e$ then
	\begin{equation} \label{MainLowerCorIneq}
		\frac{\log^2 M(P)}{Z(P)} \leq \log\|P\|_0\left(1+ \frac{2}{3}\cdot\frac{\sqrt{\log\|P\|_0}}{1-\sqrt{\log\|P\|_0}}\right).
	\end{equation}
\end{cor}

The right hand term in \eqref{MainLowerCorIneq} should be regarded as an error term as $\log\|P\|_0\to 0$, showing that $\log^2M(P)/Z(P)$ is asymptotically bounded
above by $\log\|P\|_0$.  Pritsker's aforementioned example establishes that this asymptotic formula is best possible.  Indeed, taking $P_n(z) = nz^n-1$, 
we see that $M(P) = Z(P) = n$, implying that the left hand side of \eqref{MainLowerIneq} equals $(\log^2 n)/n$.  We also observe that
\begin{equation*}
	\log\|P_n\|_0 = \log n + \frac{n(n^{-2/n} - 1)}{2} = \log n + \frac{n}{2}\left( \exp\left(\frac{-2\log n}{n}\right) - 1\right) .
\end{equation*}
Since
\begin{equation*}
	\exp\left(\frac{-2\log n}{n}\right) - 1 = \sum_{k=1}^\infty \frac{(-2)^k\log^kn}{k! n^k} = \frac{-2\log n}{n} + \frac{2\log^2 n}{n^2} + O\left(\frac{\log^3n}{n^3}\right),
\end{equation*}
we obtain that
\begin{equation*}
	\log\|P_n\|_0 = \frac{\log^2 n}{n} + O\left(\frac{\log^3n}{n^2}\right).
\end{equation*}
In other words, we have shown that
\begin{equation} \label{ExampleAsymptotic}
	\frac{\log^2 M(P_n)}{Z(P_n)} = \log\|P_n\|_0\left( 1 + O\left(\frac{\log n}{n}\right)\right).
\end{equation}
Therefore, the main term in \eqref{MainLowerIneq} provides the best possible upper bound.  We must acknowledge, however, that improvements to the error term
still seem likely.  In this example, Theorem \ref{MainLower} gives
\begin{equation*} \label{WeakPritsker}
	\frac{\log^2 M(P_n)}{Z(P_n)} \leq  \log\|P_n\|_0\left( 1 +  O\left(\frac{\log n}{\sqrt n}\right)\right),
\end{equation*}
the right hand side of which has a larger error term than that of \eqref{ExampleAsymptotic}.

In this example, all roots of $P_n$ that lie inside the unit circle have the same absolute value.  For a general polynomial $P\in\com[z]$ satisfying this property,
we will say that $P$ is {\it balanced}.  In this case, we can obtain a lower bound on $\log^2 M(P)$ analogous to Theorem \ref{MainLower}.

\begin{thm} \label{LowerBoundBalanced}
	If $P$ is a balanced polynomial over $\com$ such that $|P(0)| = 1$ and $\|P\|_0 < e$ then
	\begin{equation*}
		\frac{\log^2 M(P)}{Z(P)} \geq \log\|P\|_0.
	\end{equation*}
\end{thm}

Applying Corollary \ref{MainLowerCor} and Theorem \ref{LowerBoundBalanced} to $P_n(z) = nz^n -1$, we find that
\begin{equation*}
	\frac{\log^2 M(P_n)}{Z(P_n)} =  \log\|P_n\|_0\left( 1 +  O\left(\frac{\log n}{\sqrt n}\right)\right).
\end{equation*}
In other words, our results yield an asymptotic formula having the same main term as \eqref{ExampleAsymptotic}.  Although we cannot obtain
the best error term in this case, the advantage of Theorems \ref{MainLower} and \ref{LowerBoundBalanced} is that they apply far more generally.

\section{Proofs}

The proof of Theorem \ref{MainUpper} is rather straightforward and can be given without additional lemmas.

\begin{proof}[Proof of Theorem \ref{MainUpper}]
	We assume that
	\begin{equation*}
		P(z) = \sum_{n=0}^Na_nz^n = a_N\prod_{n=1}^N(z-z_n),
	\end{equation*}
	where $|a_0| = 1$ and observe immediately from \eqref{MeasureJensen} that
	\begin{equation*}
		\log^2 M(P) = \left( \sum_{|z_n| <1} -\log |z_n|\right)^2 \geq \sum_{|z_n|<1}\log^2 |z_n|
	\end{equation*}
	because $-\log |z_n| > 0$ for $|z_n| < 1$. By using the power series expansion for the logarithm at $1$, we find that
	\begin{align}
		\log^2 M(P) & \geq \sum_{|z_n|<1} \left( \sum_{k=1}^\infty \frac{(1-|z_n|)^k}{k}\right)^2  \nonumber \\
			& = \sum_{|z_n|<1} \sum_{k=2}^\infty \left(\sum_{\ell = 1}^{k-1} \frac{1}{\ell(k-\ell)}\right) (1-|z_n|)^k \label{*} \\
			& \geq \sum_{|z_n|<1}\left( (1-|z_n|)^2 + \sum_{k=3}^\infty \frac{(1-|z_n|)^k}{k}\right) \nonumber \\
			& = \sum_{|z_n|<1} \left(\frac{|z_n|^2-1}{2} + \sum_{k=1}^\infty \frac{(1-|z_n|)^k}{k}\right). \nonumber 
	\end{align}
	The right hand side equals $\log \|P\|_0$ from \eqref{PritskerAlpha} which completes the proof.
\end{proof}

In our proof of Theorem \ref{MainLower}, it is important to have a lower bound on the roots of $P$ in terms of the areal Mahler measure.
If $\|P\|_0$ is small, then Theorem 2.3 (a) of \cite{Pritsker} shows the roots of $P$ are close to the unit circle.  Our next lemma gives a quantitative result in this direction.

\begin{lem}\label{ElemLemma}
	Suppose that
	\begin{equation*}
		P(z) = \sum_{n=0}^Na_nz^n = a_N\prod_{n=1}^N(z-z_n),
	\end{equation*}
	where $|a_0| = 1$, and $\|P\|_0<e$.   Then $|z_n| \geq 1-\sqrt{\log\|P\|_0}$ for all $n$.
\end{lem}
\begin{proof}
	Let $f(x) = (x^2-1)/2 -\log x$.  Since $f(x)>0$ is decreasing on $(0,1)$, we find  from \eqref{PritskerAlpha} that
	\begin{equation} \label{SumToMax}
		\log\|P\|_0 = \sum_{|z_n|<1} f(|z_n|) \geq \max_{|z_n|<1}\{f(|z_n|)\} = f(\min_{|z_n|<1} |z_n|).
	\end{equation}
	By writing $f$ in its power series expansion at $1$, we find that $f(x) \geq (1-x)^2$ for all $x\in(0,1)$.
	Combing this with the above inequality, we obtain that
	\begin{equation*}
		\min_{|z_n|<1} |z_n| \geq 1- \sqrt{\log\|P\|_0}
	\end{equation*}
	completing the proof.
\end{proof}
	
Assume that $P_N\in\com[z]$ is any sequence of polynomials of degree $N$, each having roots $z_{N,1}, z_{N,2},\ldots,z_{N,N}$.
If $\lim_{N\to\infty}\log\|P_N\|_0 = 0$, then Pritsker showed in Theorem 2.3 (a) of \cite{Pritsker} that
\begin{equation*}
	\liminf_{N\to\infty}\min_{1\leq n\leq N}|z_{N,n}| \geq 1.
\end{equation*}
The special case involving polynomials with integers coefficients is a consequence of Lemma \ref{ElemLemma}.  Now we present our proof of Theorem \ref{MainLower}.

\begin{proof}[Proof of Theorem \ref{MainLower}]
	Applying Cauchy's inequality, we obtain immediately that
	\begin{equation} \label{FirstHolder}
		\log^2 M(P) \leq Z(P)\cdot \sum_{|z_n|<1}\log^2 |z_n|.
	\end{equation}
	By writing the logarithm in its power series expansion as in the proof of Theorem \ref{MainUpper}, we have that
	\begin{align*} \label{LogPowerSeries}
		\sum_{|z_n|<1}\log^2 |z_n| & = \sum_{|z_n|<1} \sum_{k=2}^\infty \left(\sum_{\ell = 1}^{k-1} \frac{1}{\ell(k-\ell)}\right) (1-|z_n|)^k \\
			& = \sum_{|z_n|<1}\left( (1-|z_n|)^2 + \sum_{k=3}^\infty \left(\sum_{\ell = 1}^{k-1} \frac{1}{\ell(k-\ell)}\right) (1-|z_n|)^k\right) \\
			& = \sum_{|z_n|<1}\left( (1-|z_n|)^2 + \sum_{k=3}^\infty\frac{(1-|z_n|)^k}{k} + \sum_{k=3}^\infty \left(\sum_{\ell = 1}^{k-1} \frac{1}{\ell(k-\ell)} - \frac{1}{k}\right) (1-|z_n|)^k\right).
	\end{align*}
	It is easily checked that
	\begin{equation} \label{PowerSeries}
		(1-x)^2 + \sum_{k=3}^\infty \frac{(1-x)^k}{k} = \frac{x^2-1}{2} + \sum_{k=1}^\infty \frac{(1-x)^k}{k} = \frac{x^2-1}{2} - \log x
	\end{equation}
	for all $0<x\leq 1$, which combined with \eqref{PritskerAlpha}, yields
	\begin{align} \label{RepBound}
		\sum_{|z_n|<1}\log^2 |z_n| & = \sum_{|z_n|<1}\left( \frac{|z_n|^2-1}{2} - \log|z_n|+ \sum_{k=3}^\infty \left(\sum_{\ell = 1}^{k-1} \frac{1}{\ell(k-\ell)} - \frac{1}{k}\right) (1-|z_n|)^k\right) \nonumber \\
			& = \log\|P\|_0 + \sum_{|z_n|<1} \sum_{k=3}^\infty \left(\sum_{\ell = 1}^{k-1} \frac{1}{\ell(k-\ell)} - \frac{1}{k}\right) (1-|z_n|)^k \nonumber \\
			& =  \log\|P\|_0 + \sum_{|z_n|<1} (1-|z_n|)^2\sum_{k=1}^\infty \left(\sum_{\ell = 1}^{k+1} \frac{1}{\ell(k+2-\ell)} - \frac{1}{k+2}\right) (1-|z_n|)^k.
	\end{align}
	Now let $g:(0,1)\to \real$ be given by
	\begin{equation*}
		g(x) = \sum_{k=1}^\infty \left(\sum_{\ell = 1}^{k+1} \frac{1}{\ell(k+2-\ell)} - \frac{1}{k+2}\right) (1-x)^k
	\end{equation*}
	Since $g'(x) < 0$
	for all $x\in (0,1)$, we know that $g$ is decreasing on this interval, and Lemma \ref{ElemLemma} implies that $g(|z_n|) \leq g(1-\sqrt{\log\|P\|_0})$.
	Using this observation in \eqref{RepBound}, we see that
	\begin{align*}
		\sum_{|z_n|<1}\log^2 |z_n| & \leq  \log\|P\|_0 + g(1-\sqrt{\log\|P\|_0})\sum_{|z_n|<1} (1-|z_n|)^2 \\
			& \leq  \log\|P\|_0 + g(1-\sqrt{\log\|P\|_0})\sum_{|z_n|<1} \left((1-|z_n|)^2 + \sum_{k=3}^\infty \frac{(1-|z_n|)^k}{k}\right) \\
			& \leq  \log\|P\|_0 + g(1-\sqrt{\log\|P\|_0})\log\|P\|_0
	\end{align*}
	and the result follows from \eqref{FirstHolder}.
\end{proof}

Although the proof of Theorem \ref{MainLower} uses several estimates, we believe the weakest of these are the use of Cauchy's Inequality and Lemma \ref{ElemLemma}.
Our application of Cauchy's Inequality in \eqref{FirstHolder} admits equality whenever the values $|z_n|$ are all equal (i.e., whenever $P$ is balanced).  
On the other hand, inequality \eqref{SumToMax}
in the proof of Lemma \ref{ElemLemma} is sharpest when one root of $P$ is somewhat far from the unit circle while the others are close.  To summarize, these estimates
are sharpest in opposing cases, so that their simultaneous use leads to a weaker estimate that should be expected.

Nevertheless, Lemma \ref{ElemLemma} is used only to estimate the error term in Theorem \ref{MainLower}.  In Pritsker's sequence $P_n(z) = nz^n-1$, all roots of $P_n$
have the same absolute value, meaning that we have equality in \eqref{FirstHolder}.  Indeed, we found that
\begin{equation*}
	\frac{\log^2 M(P_n)}{Z(P_n)} = \log\|P_n\|_0\left( 1 + O\left(\frac{\log n}{n}\right)\right),
\end{equation*}
while Theorem \ref{MainLower} gives
\begin{equation*}
	\frac{\log^2 M(P_n)}{Z(P_n)} \leq  \log\|P_n\|_0\left( 1 +  O\left(\frac{\log n}{\sqrt n}\right)\right).
\end{equation*}
As expected, these estimates have the same main term, while differing in their error terms.

\begin{proof}[Proof of Corollary \ref{MainLowerCor}]
	It is easily verified that $\sum_{\ell=1}^{k-1}1/(\ell(k-\ell)) - 1/k \leq 2/3$ for all $k\geq 3$.  Therefore, Theorem \ref{MainLower} yields
	\begin{equation*}
		\frac{\log^2 M(P)}{Z(P)} \leq \log\|P\|_0\left(1+ \frac{2}{3}\sum_{k=1}^\infty\log^{k/2}\|P\|_0\right)
	\end{equation*}
	and the result follows by summing the geometric series.
 \end{proof}
 
 \begin{proof}[Proof of Theorem \ref{LowerBoundBalanced}]
 	Since all roots of $P$ have the same absolute value, we obtain equality in \eqref{FirstHolder} giving
	\begin{equation*}
		\log^2 M(P) = Z(P)\cdot \sum_{|z_n|<1}\log^2 |z_n|.
	\end{equation*}
	Therefore, we may apply \ref{RepBound} to obtain that
	\begin{equation*}
		\frac{\log^2 M(P)}{Z(P)} = \log\|P\|_0 + \sum_{|z_n|<1} (1-|z_n|)^2\sum_{k=1}^\infty \left(\sum_{\ell = 1}^{k+1} \frac{1}{\ell(k+2-\ell)} - \frac{1}{k+2}\right) (1-|z_n|)^k.
	\end{equation*}
	It is straightforward to verify that
	\begin{equation*}
		\sum_{\ell = 1}^{k+1} \frac{1}{\ell(k+2-\ell)} - \frac{1}{k+2} \geq 0
	\end{equation*}
	for all $k\geq 1$, so the result follows.
	
 \end{proof}

\end{document}